\documentclass[12pt]{article}
\usepackage{amsfonts}
\usepackage{amsmath,latexsym,amssymb,amsfonts,amsbsy, amsthm}
\usepackage{graphicx}
\usepackage{amsthm}
\usepackage[usenames]{color}
\usepackage{xspace,colortbl}
\allowdisplaybreaks

\newcommand{\beq}{\begin{equation}}
\newcommand{\eeq}{\end{equation}}
\newcommand{\ben}{\begin{eqnarray}}
\newcommand{\een}{\end{eqnarray}}
\newcommand{\beno}{\begin{eqnarray*}}
\newcommand{\eeno}{\end{eqnarray*}}

\newtheorem{thm}{Theorem}[section]

\newtheorem{lem}[thm]{Lemma}

\newtheorem{rmk}[thm]{Remark}

\begin{document}

\title{\Large \bf On the Lane-Emden conjecture}

\author{Kui Li$^{a}$ ~~ Zhitao Zhang$^{b,c,}$\thanks{Corresponding author, supported by NSF of China (No. 11325107,11771428)}\\
{\small $^{a}$ School of Mathematics and Statistics, Zhengzhou University,
Zhengzhou,}\\
{\small Henan 450001, P. R. China. E-mail: likui@zzu.edu.cn}\\
{\small $^{b}$  Academy of Mathematics and Systems Science, Chinese Academy of }\\
{\small  Sciences, Beijing 100190;}\\
{\small  $^{c}$   School of Mathematical Sciences,~ University of Chinese Academy of }\\
{\small Sciences, Beijing 100049,  P. R. China. Email: zzt@math.ac.cn}}
\date{}
\maketitle

\renewcommand{\theequation}{\thesection.\arabic{equation}}
\setcounter{equation}{0}

\begin{abstract}
We consider the Lane-Emden conjecture which states that there is no non-trivial non-negative solution for the Lane-Emden system whenever the pair of exponents is subcritical. By Sobolev embeddings on $S^{N-1}$ and scale invariance of the solutions, we show this conjecture holds in a new region. Our methods can also be used to prove the Lane-Emden conjecture in space dimension $N\leq4$, that is to give a different proof of the main result in \cite{S}.
\end{abstract}

\section{Introduction}
\par Liouville-type results for the nonlinear equation
\begin{equation}\label{leequ}
 -\Delta u=u^p,~~~x\in \mathbb{R}^N
\end{equation}
have received considerable  attention and have many applications both for elliptic and parabolic equations (see for instance \cite{BM}, \cite{F}, \cite{FY}, \cite{GS}, \cite{QS}, \cite{Z}).\par
 Let
\begin{equation*}
 p_S=\left\{
\begin{aligned}
 &\infty,~~~&\mbox{if}~~N=2, \\
 &(N+2)/(N-2),~~~&\mbox{if}~~N\geq3.
\end{aligned}
\right.
\end{equation*}
Gidas and Spruck \cite{GS} showed that equation \eqref{leequ} had no positive classical solutions if $1<p<p_S$. Later, Chen and Li \cite{CL1} gave a new proof by the Kelvin transform and the method of moving planes.\par
The natural extension of equation \eqref{leequ} is the following Lane-Emden system
\begin{equation}\label{lesys}
\left\{ \begin{aligned}
 &-\Delta u=v^p,~~~ & x\in \mathbb{R}^N,\\
 &-\Delta v=u^q,~~~ & x\in \mathbb{R}^N.
  \end{aligned} \right.
\end{equation}\par
For given positive constants $p$ and $q$, we call the pair $(p,q)$ subcritical if the pair $(p,q)$ lies below the Sobolev hyperbola, i.e.
\begin{equation}\label{subcri}
p>0,~q>0~\mbox{and}~\frac{1}{p+1}+\frac{1}{q+1}>1-\frac{2}{N}.
\end{equation}
If $pq\neq1$, we define
\begin{equation}\label{alphabeta}
\alpha=\frac{2(p+1)}{pq-1},~~~\beta=\frac{2(q+1)}{pq-1}.
\end{equation}\par
Suppose that $pq\neq1$ and $(u,v)$ is a solution of system \eqref{lesys}.  Then $(R^\alpha u(Rx), R^\beta v(Rx))$ is also a solution of system \eqref{lesys} for any $R>0$. This scale invariance is very important in our proofs.\par
Mitidieri \cite{M} showed that the system \eqref{lesys} had no positive radial solutions if and only if pair $(p,q)$ was subcritical, which implies the following  conjecture (see \cite{FF}, \cite{FG}, \cite{SZ1}) is true for the positive radial solutions.\par
\textbf{Lane-Emden Conjecture.} \emph{If the pair $(p,q)$ is subcritical, then system \eqref{lesys}
has no positive classical solutions.}\par

Serrin and Zou \cite{SZ1} showed this conjecture held if $pq\leq1$ or $pq>1$ and $\max\{\alpha,\beta\}\geq N-2$. Pol\'{a}\v{c}ik  et al. \cite{PQS} established a Doubling Lemma and showed that if $pq>1$, then the non-existence of  bounded positive classical solutions for system \eqref{lesys} implied the non-existence of positive classical solutions, and proved this conjecture for dimension $N=3$. Souplet \cite{S} proved this conjecture for dimension $N=4$.\par
By Sobolev embeddings on $S^{N-1}$ and scale invariance of the solutions, we show this conjecture holds in a new region and have the following theorem.
\begin{thm}\label{mainthm1}
Suppose  that the pair $(p,q)$ is subcritical. If $\min\{p,q\}\leq1$, then system \eqref{lesys} has no positive classical solutions.
\end{thm}
Our methods can be used to give a different proof of the following theorem which is the main result in \cite{S}.
\begin{thm}\label{mainthm}
Suppose that the pair $(p,q)$ is subcritical. If $pq\leq1$ or $pq>1$ and $\max\{\alpha,\beta\}>N-3$, then system \eqref{lesys} has no positive classical solutions.
\end{thm}
\begin{center}
\includegraphics[height=2.5in]{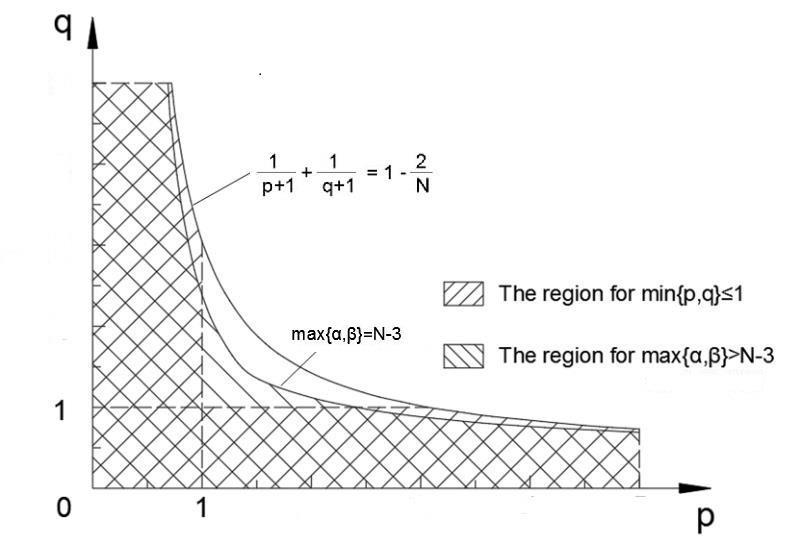}\\
Figure~1.~~The~regions~obtained in Theorem \ref{mainthm1} and \ref{mainthm} when $N=8$.
\end{center}

\begin{rmk}
Suppose that $N\geq 5$ and $u$ is a positive solution of the fourth order equation
\begin{equation}\label{leequ1}
\Delta^2 u=u^p,~~~x\in \mathbb{R}^N
\end{equation}
with $1<p<\frac{N+4}{N-4}$. Then $-\Delta u>0$ (see \cite{WX}). Hence $(u,v)$ with $v=-\Delta u$ is a positive solution of system \eqref{lesys} with $q=1$. Therefore, Theorem \ref{mainthm1} covers Theorem 1.4 \cite{L} and partial results of Theorem 1.4 \cite{WX}.\par
\end{rmk}

Let $B_R=\{x\in\mathbb{R}^N: |x|<R\}$, $S_R=\{x\in\mathbb{R}^N: |x|=R\}$ and $S^{N-1}=S_1$. For $0\neq x\in \mathbb{R}^N $, let $\theta=\frac{x}{|x|} \in S^{N-1}$ and $r=|x|$.  For $k\in [1,\infty]$, $w\in L^k(S^{N-1})$, we set $||w||_k=||w||_{L^k(S^{N-1})}$. $C(N,p,q,\cdots)$, $a(N,p,q,\cdots)$ and $b(N,p,q,\cdots)$ denote positive constants which are dependent on $N$, $p$, $q$, $\cdots$ and independent of the solution $(u,v)$ of system \eqref{lesys}.\par

This paper is organized as follows: in section 2, we give some preliminaries; in section 3, we prove Theorem \ref{mainthm1}; and in section 4, we prove Theorem \ref{mainthm}.

\section{ Preliminaries}
\numberwithin{equation}{section}
 \setcounter{equation}{0}
Without loss of generality,  we assume that $p\geq q>0$. For a continuous function $u(x)$ in $\mathbb{R}^N$, we write $u=u(r,\theta)$ and  define
$$\overline{u}(r)=\frac{1}{\omega_N}\int_{S^{N-1}}u(r,\theta)d\theta,$$
where $\omega_N=|S^{N-1}|$ and $d\theta$ is the surface measure on $S^{N-1}$. On $S^{N-1}$, we have the following Sobolev embeddings:
\begin{lem}\label{Sobolevimbedings}(Sobolev embeddings on $S^{N-1}$). Let $N\geq2$ and $j\geq1$ be integers, $k\in(1,\infty]$ and $w=w(\theta)\in W^{j,k}(S^{N-1})$.
\begin{description}
  \item[(1)] If $k<\frac{N-1}{j}$, then $||w||_\lambda\leq C(||D_\theta^jw||_k+||w||_1)$, where $C=C(j,k,n)>0$ and $\lambda$ satisfies
  $$\frac{1}{k}-\frac{1}{\lambda}=\frac{j}{N-1}.$$
  \item[(2)]  If $k=\frac{N-1}{j}$, then for any $\lambda\in[1,\infty)$, $||w||_\lambda\leq C(||D_\theta^jw||_k+||w||_1)$ with $C=C(\lambda,k,n)>0$.
  \item[(3)] If $k>\frac{N-1}{j}$, then $||w||_{\infty}\leq C(||D_\theta^jw||_k+||w||_1)$ with $C=C(j,k,n)>0$.
\end{description}
\end{lem}
For the proof of this lemma, see Lemma 2.1 \cite{LZ}.

\begin{lem}\label{lpesti}(See \cite{CHL}, Lemma 2.2). Let $k\in (1,\infty)$ and $R>0$. For $u\in W^{2,k}(B_{2R})$, we have
\begin{equation*}
\|D^2 u\|_{L^k(B_R)}\leq C(\|\Delta u\|_{L^k(B_{2R})}+R^{\frac{N}{k}-(N+2)}\|u\|_{L^1(B_{2R})}),
\end{equation*}
where $C=C(k,N)>0$.
\end{lem}
\begin{lem} \label{esti} Suppose that $pq>1$ and $(u,v)$ is a positive solution of system \eqref{lesys}. Then for any $r>0$,
\begin{equation*}
\overline{u}(r)\leq C r^{-\alpha},~\overline{v}(r)\leq C r^{-\beta},~\int_{B_r}u\leq C r^{N-\alpha}~\mbox{and}~\int_{B_r}v\leq C r^{N-\beta},
\end{equation*}
where $C=C(p,q,N)>0$ and $\alpha$, $\beta$ are defined in \eqref{alphabeta}.
\end{lem}
This Lemma can be obtained by Proposition 2.1, Corollary 2.1 and Theorem 3.1 \cite{CHL} immediately.

\begin{lem} \label{compair} Suppose that $pq>1$ and $(u,v)$ is a bounded positive solution of system \eqref{lesys}. Then
\begin{equation*}
\frac{v^{p+1}}{p+1}\leq\frac{u^{q+1}}{q+1},~~~\forall~x\in~\mathbb{R}^N.
\end{equation*}
\end{lem}
For the proof, see \cite{CHL} or \cite{S}. Chen and Li \cite{CL2} proved the following result:
\begin{lem} \label{integralsys} Suppose that $(u,v)$ is a positive solution of system \eqref{lesys}. Then $(u,v)$ solves the integral system
\begin{equation}\label{integralsystem}
\left\{ \begin{aligned}
 & u(x)=\frac{1}{(N-2)\omega_N}\int_{\mathbb{R}^N}\frac{v^p(y)}{|x-y|^{N-2}}dy,~x\in\mathbb{R}^N\\
 & v(x)=\frac{1}{(N-2)\omega_N}\int_{\mathbb{R}^N}\frac{u^q(y)}{|x-y|^{N-2}}dy,~x\in\mathbb{R}^N.
  \end{aligned} \right.
\end{equation}
\end{lem}
In order to prove our theorems, we need the following energy type inequality.
\begin{lem}\label{ph}
Suppose that $(u,v)$ is a non-negative solution of system \eqref{lesys}. Then for any $R>0$ and $\lambda\in\mathbb{R}$, there holds
\begin{equation*}
\begin{aligned}
&(\frac{N}{q+1}-\frac{N-2}{2}-\lambda)\int_{B_R}u^{q+1}+(\frac{N}{p+1}-\frac{N-2}{2}+\lambda)\int_{B_R}v^{p+1}\\
\leq&R\int_{S_R}(\frac{u^{q+1}}{q+1}+\frac{v^{p+1}}{p+1})]+\lambda \int_{S_R}(\frac{\partial v}{\partial\nu}u-\frac{\partial u}{\partial\nu}v).
\end{aligned}
\end{equation*}
\end{lem}
\begin{proof}
Let $H(x,y)=u^q(x)v^p(y)$. By Lemma \ref{integralsys}, we have
\begin{equation*}
\left\{ \begin{aligned}
 & u(x)=\frac{1}{(N-2)\omega_N}\int_{\mathbb{R}^N}\frac{v^p(y)}{|x-y|^{N-2}}dy,~x\in\mathbb{R}^N\\
 & v(x)=\frac{1}{(N-2)\omega_N}\int_{\mathbb{R}^N}\frac{u^q(y)}{|x-y|^{N-2}}dy,~x\in\mathbb{R}^N.
  \end{aligned} \right.
\end{equation*}
Hence
\begin{equation*}
\left\{ \begin{aligned}
 & \nabla u(x)\cdot x=-\frac{1}{\omega_N}\int_{\mathbb{R}^N}\frac{x(x-y)|v^p(y)}{|x-y|^N}dy,\\
 & \nabla v(x)\cdot x=-\frac{1}{\omega_N}\int_{\mathbb{R}^N}\frac{x(x-y)|u^q(y)}{|x-y|^N}dy.
  \end{aligned} \right.
\end{equation*}
On one hand,
\begin{equation*}
\int_{B_R} u^qDu\cdot x=\int_{B_R} D(\frac{u^{q+1}}{q+1})\cdot x=\frac{R}{q+1}\int_{S_R}u^{q+1}-\frac{N}{q+1}\int_{B_R}u^{q+1}.
\end{equation*}
On the other hand,
\begin{equation*}
\begin{aligned}
&\int_{B_R} u^qDu\cdot x=-\frac{1}{\omega_N}\int_{B_R}dx\int_{\mathbb{R}^N}\frac{x(x-y)H(x,y)}{|x-y|^N}dy\\
=&-2\frac{1}{2\omega_N}\int_{B_R}dx\int_{\mathbb{R}^N}\frac{x(x-y)H(x,y)}{|x-y|^N}dy\\
=&-\frac{1}{2\omega_N}\int_{B_R}dx\int_{\mathbb{R}^N}\frac{x(x-y)H(x,y)}{|x-y|^N}dy+\frac{2-N}{2}\int_{B_R} u^{q+1}\\
&~~~~~~~~~~~~-\frac{1}{2\omega_N}\int_{B_R}dx\int_{\mathbb{R}^N}\frac{y(x-y)H(x,y)}{|x-y|^N}dy.
\end{aligned}
\end{equation*}
Hence
\begin{equation}\label{ph1}
\begin{aligned}
&(\frac{N}{q+1}-\frac{N-2}{2})\int_{B_R} u^{q+1}=\frac{R^{b+1}}{q+1}\int_{S_R}u^{q+1}\\
+&\frac{1}{2\omega_N}\int_{B_R}dx\int_{\mathbb{R}^N}\frac{(|x|^2-|y|^2)H(x,y)}{|x-y|^{N-2m+2}}dy.
\end{aligned}
\end{equation}
Similarly, we have
\begin{equation}\label{ph2}
\begin{aligned}
&(\frac{N}{p+1}-\frac{N-2}{2})\int_{B_R} v^{p+1}=\frac{R}{p+1}\int_{S_R}v^{p+1}\\
+&\frac{1}{2\omega_N}\int_{B_R}dx\int_{\mathbb{R}^N}\frac{(|x|^2-|y|^2)H(y,x)}{|x-y|^N}dy.
\end{aligned}
\end{equation}
Therefore, by \eqref{ph1} and \eqref{ph2}, we have
\begin{equation}\label{phh}
\begin{aligned}
&(\frac{N}{q+1}-\frac{N-2}{2})\int_{B_R}u^{q+1}+(\frac{N}{p+1}-\frac{N-2}{2})\int_{B_R}v^{p+1}\\
=&\frac{R^{1+b}}{q+1}\int_{S_R}u^{q+1}+\frac{R^{1+a}}{p+1}\int_{S_R}v^{p+1}+I\\
\leq&\frac{R^{1+b}}{q+1}\int_{S_R}u^{q+1}+\frac{R^{1+a}}{p+1}\int_{S_R}v^{p+1},
\end{aligned}
\end{equation}
where $I=\frac{1}{2\omega_N}\int_{B_R}dx\int_{B_R^c}\frac{(|x|^2-|y|^2)[H(x,y)+H(y,x)]}{|x-y|^{N-2m+2}}dy\leq 0$. \par
Multiplying the first equation of system \eqref{lesys} by $v$, the second equation by $u$ and integrating over $B_R$, we have
\begin{equation*}
\left\{ \begin{aligned}
 &\int_{B_R}DuDv-\int_{B_R} v^{p+1}=\int_{S_R}\frac{\partial u}{\partial\nu}v,\\
 &\int_{B_R}DuDv-\int_{B_R} u^{q+1}=\int_{S_R}\frac{\partial v}{\partial\nu}u,\\
  \end{aligned} \right.
\end{equation*}
hence
\begin{equation}\label{ph3}
\int_{B_R} u^{q+1}-\int_{B_R} v^{p+1}=\int_{S_R}(\frac{\partial u}{\partial\nu}v-\frac{\partial v}{\partial\nu}u).
\end{equation}
By \eqref{phh} and \eqref{ph3}, we get this lemma.
\end{proof}

\section{Proof of Theorem \ref{mainthm1}}
\numberwithin{equation}{section}
 \setcounter{equation}{0}
In this section, we prove Theorem \ref{mainthm1}. Let $\kappa=\frac{\alpha+\beta}{\alpha}$.
\begin{lem}\label{decayfore}
Suppose that $\gamma \in(0,1)$ and $(u,v)$ is a positive solution of system \eqref{lesys}. Then there exists a positive constant $C=C(N,p,q,\gamma)$ such that
\begin{equation*}
\int_{B_2}|\nabla u^\frac{\gamma}{2}|^2\leq C.
\end{equation*}
\end{lem}
\begin{proof}
Let  $\gamma \in(0,1)$ and $(u,v)$ be a positive solution of system \eqref{lesys}. By Lemma 2.2 of \cite{SZ1}, for any $\varphi\in C_c^\infty(\mathbb{R}^N)$, we have
\begin{equation}\label{decayfore1}
\int_{\mathbb{R}^N}\varphi^2 u^{\gamma-2}|\nabla u|^2\leq C\int_{\mathbb{R}^N}|\nabla \varphi|^2 u^\gamma
\end{equation}
with $C=C(N,\gamma)>0$. Let $\eta(x)\in C_c^\infty(\mathbb{R}^N)$ be a cut-off function:
\begin{equation*}\eta(x)=
\begin{cases} 1, & |x|<2,\\
  0, & |x|>4.
\end{cases}
\end{equation*}
Let $\varphi=\eta$ in \eqref{decayfore1}. Then by \eqref{decayfore1}, Lemma \ref{esti} and Jensen's inequality, we have
\begin{equation*}
\int_{B_2}|\nabla u^\frac{\gamma}{2}|^2\leq C\int_{2\leq|x|\leq 4}u^\gamma\leq C\int_2^4r^{N-1}\overline{u}^\gamma(r)\leq C\int_2^4r^{N-1}r^{-\alpha r}\leq C,
\end{equation*}
where $C=C(N,p,q,\gamma)>0$.
\end{proof}

\begin{lem}\label{decayfore101}
Suppose that $pq>1$ and $(u,v)$ is a bounded positive solution of system \eqref{lesys}. Then
\begin{equation*}
\int_{B_2}|\nabla u^{\frac{\kappa}{2}}|^2\leq C[\int_{B_4}u^{q+1}+1],
\end{equation*}
where $C=C(N,p,q)>0$.
\end{lem}
\begin{proof}
Let $\eta(x)$ be the cut-off function in the proof of Lemma \ref{decayfore}.
Multiplying the first equation of system \eqref{lesys} by $u^{\kappa-1}\eta$ and integrating over $\mathbb{R}^N$, we have
\begin{equation}\label{decayfore1011}
(\kappa-1)\int_{B_4} u^{\kappa-2}|\nabla u|^2\eta=\int_{B_4} u^{\kappa-1}v^p\eta+\frac{1}{\kappa}\int_{B_4}u^\kappa \Delta \eta.
\end{equation}
Since $\kappa=\frac{\alpha+\beta}{\alpha}=1+\frac{q+1}{p+1}$ and $pq>1$, then $\kappa<q+1$, and by Lemma \ref{compair}, we have
\begin{equation}\label{decayfore1012}
u^{\kappa-1}v^p\leq (\frac{q+1}{p+1})^{\frac{p}{p+1}}u^{\kappa-1} u^{\frac{p(q+1)}{p+1}}=(\frac{q+1}{p+1})^{\frac{p}{p+1}}u^{q+1}.
\end{equation}
Hence by \eqref{decayfore1011}, \eqref{decayfore1012} and H\"{o}lder inequlity, we have
\begin{equation}\label{decayfore1013}
\begin{aligned}
\int_{B_2}|\nabla u^{\frac{\kappa}{2}}|^2&\leq\int_{B_4}|\nabla u^{\frac{\kappa}{2}}|^2\eta\\
&\leq C\int_{B_4} u^{\kappa-2}|\nabla u|^2\eta\\
&\leq C\int_{B_4} (u^{\kappa-1}v^p+u^\kappa)\\
&\leq C\int_{B_4} u^{q+1}+C(\int_{B_4} u^{q+1})^\frac{\kappa}{q+1}\\
&\leq C[\int_{B_4} u^{q+1}+1],
\end{aligned}
\end{equation}
where $C=C(N,p,q)>0$.
\end{proof}

\begin{lem}\label{decayfore102}
Suppose that $\gamma \in(0,1)$,  $pq>1$ and $(u,v)$ is a bounded positive solution of system \eqref{lesys}. Then there exists $r_0\in[1,2]$ such that
\begin{equation*}
\int_{S_{r_0}}|\nabla u|\leq C,\int_{S_{r_0}}|\nabla v|\leq C,\int_{S_{r_0}}|\nabla u^\frac{\gamma}{2}|^2\leq C,\int_{S_{r_0}}|\nabla u^{\frac{\kappa}{2}}|^2\leq C[\int_{B_4}u^{q+1}+1],
\end{equation*}
and
\begin{equation*}
\int_{S_{r_0}}|D^2 u|^{\frac{p+1}{p}}\leq C\int_{B_2}|D^2 u|^{\frac{p+1}{p}},\int_{S_{r_0}}|D^2 v|^{\frac{q+1}{q}}\leq C\int_{B_2}|D^2 u|^{\frac{q+1}{q}},
\end{equation*}
where $C=C(N,p,q,\gamma)>0$.
\end{lem}
\begin{proof}
Let $\gamma \in(0,1)$. By Lemma \ref{esti}, Lemma \ref{decayfore} and Lemma \ref{decayfore101} and interpolation inequalities, there exists  $C_1=C_1(N,p,q,\gamma)>0$ such that
\begin{equation*}
\begin{aligned}
&\int_{B_2}|\nabla u|\leq C(N)(\int_{B_4}|\Delta u|+\int_{B_4}|u|)\leq C_1,\\
&\int_{B_2}|\nabla v|\leq C(N)(\int_{B_4}|\Delta v|+\int_{B_4}|v|)\leq C_1
\end{aligned}
\end{equation*}
and
\begin{equation*}
\int_{B_2}|\nabla u^\frac{\gamma}{2}|^2\leq C_1,~~\int_{B_2}|\nabla u^{\frac{\kappa}{2}}|^2\leq C_1[\int_{B_4}u^{q+1}+1].
\end{equation*}
\par Define
\begin{equation*}
\begin{aligned}
&A=\{r\in[1,2]:\int_{S_r}|\nabla u|>7 C_1\},\\
&B=\{r\in[1,2]:\int_{S_r}|\nabla v|>7 C_1\},\\
&C=\{r\in[1,2]:\int_{S_r}|\nabla u^\frac{\gamma}{2}|^2>7 C_1\},\\
&D=\{r\in[1,2]:\int_{S_r}|\nabla u^{\frac{\kappa}{2}}|^2>7C_1[\int_{B_4}u^{q+1}+1],\\
&E=\{r\in[1,2]:\int_{S_r}|D^2 u|^{\frac{p+1}{p}}>7\int_{B_2}|D^2 u|^{\frac{p+1}{p}}\},\\
&F=\{r\in[1,2]:\int_{S_r}|D^2 v|^{\frac{q+1}{q}}>7\int_{B_2}|D^2 v|^{\frac{q+1}{q}}\}.
\end{aligned}
\end{equation*}
Then
\begin{equation*}
7C_1 |A|\leq\int_A\int_{S_r}|\nabla u|\leq\int_1^2\int_{S_r}|\nabla u|\leq\int_{B_2}|\nabla u|\leq C_1,
\end{equation*}
hence
\begin{equation*}
|A|\leq\frac{1}{7},
\end{equation*}
where $|A|$ is the Lebesgue measure of $A$. Similarly, we have
$$|B|\leq\frac{1}{7},~~|C|\leq\frac{1}{7},~~|D|\leq\frac{1}{7},~~|E|\leq\frac{1}{7}~~\mbox{and}~~|F|\leq\frac{1}{7}.$$
Hence
$$|A\cup B\cup C \cup D \cup E\cup F|\leq\frac{6}{7}~~\mbox{and}~~|A^c\cap B^c\cap C^c\cap D^c \cap E^c\cap F^c|\geq\frac{1}{7},$$
where $G^c=[1,2]-G$ for $G\subseteq[1,2]$. Let $r_0\in A^c\cap B^c\cap C^c \cap D^c \cap E^c\cap F^c$. Then we obtain this lemma with $C=7\max\{C_1,1\}$.
\end{proof}

We shall fix this specific $r_0$ in the rest of the proof and define $F(r)=\int_{B_r}u^{q+1}$.
\begin{lem}\label{decayforgrand}
Suppose that  $pq>1$ and $(u,v)$ is a bounded positive solution of system \eqref{lesys}. Then there exists $C=C(N,p,q)>0$ such that
\begin{equation*}
(\int_{S_{r_0}}|Du|^{\frac{p+1}{p}})^{\frac{p}{p+1}}\leq C[F(4)^{\frac{p}{p+1}}+1],~~~(\int_{S_{r_0}}|Dv|^{\frac{q+1}{q}})^{\frac{q}{q+1}}\leq C[F(4)^{\frac{q}{q+1}}+1].
\end{equation*}
\end{lem}
\begin{proof}
By Lemma \ref{Sobolevimbedings}, Lemma \ref{lpesti}, Lemma \ref{esti}, Lemma \ref{compair}, Lemma \ref{decayfore102}, Jensen's inequality and the fact that $r_0\in[1,2]$, we have
\begin{equation*}
\begin{aligned}
(\int_{S_{r_0}}|Du|^{\frac{p+1}{p}})^{\frac{p}{p+1}}&\leq C(\int_{S_1}|D u(r_0,\theta)|^{\frac{p+1}{p}})^{\frac{p}{p+1}}\\
&\leq C[(\int_{S_1}|D_\theta Du(r_0,\theta)|^{\frac{p+1}{p}})^{\frac{p}{p+1}}+\int_{S_{r_0}}|Du|]\\
&\leq C[(\int_{S_{r_0}}|D^2u|^{\frac{p+1}{p}})^{\frac{p}{p+1}}+1]\\
&\leq C[(\int_{B_2}|D^2u|^{\frac{p+1}{p}})^{\frac{p}{p+1}}+1]\\
&\leq C[(\int_{B_4}|\Delta u|^{\frac{p+1}{p}})^{\frac{p}{p+1}}+1]\\
&\leq C[(\int_{B_4}v^{p+1})^{\frac{p}{p+1}}+1]\\
&\leq C[(\int_{B_4}u^{q+1})^{\frac{p}{p+1}}+1]= C[F(4)^{\frac{p}{p+1}}+1],
\end{aligned}
\end{equation*}
where $C=C(N,p,q)>0$. Similarly, we have
\begin{equation*}
(\int_{S_{r_0}}|Dv|^{\frac{q+1}{q}})^{\frac{q}{q+1}}\leq C[F(4)^{\frac{q}{q+1}}+1]
\end{equation*}
with $C=C(N,p,q)>0$.
\end{proof}

\begin{lem}\label{decayforu}
Suppose that $q\leq1$, $pq>1$, $(p,q)$ is subcritical and $(u,v)$ is a bounded positive solution of system \eqref{lesys}. Then there exist $C=C(N,p,q)>0$ and $a=a(N,p,q)\in [0,1)$ such that
\begin{equation*}
\int_{S_{r_0}}u^{q+1}\leq C[F(4)^a+1].
\end{equation*}
\end{lem}
\begin{proof}
By Lemma \ref{Sobolevimbedings}, Lemma \ref{esti}, Lemma \ref{decayfore102} and and the fact that $r_0\in[1,2]$, if $N\leq 3$, then we have
\begin{equation}\label{decayforu11}
\int_{S_1}u^{q+1}(r_0,\theta)d\theta\leq C
\end{equation}
and if $N\geq 4$, then we have
\begin{equation}\label{decayforu1}
\int_{S_1}u^{\frac{\gamma(N-1)}{N-3}}(r_0,\theta)d\theta\leq C,
\end{equation}
where $\gamma\in(0,1)$ and $C=C(N,p,q,\gamma)$.
Suppose that $q<\frac{2}{N-3}$. Then there exists $\gamma_0\in (0,1)$ such that
\begin{equation*}
q+1<\frac{\gamma_0(N-1)}{N-3}
\end{equation*}
and
\begin{equation}\label{decayforu12}
\int_{S_{r_0}}u^{q+1}\leq 2^{N-1}\int_{S_1}u^{q+1}(r_0,\theta)\leq2^{N-1}[\int_{S_1}u^{\frac{\gamma_0(N-1)}{N-3}}(r_0,\theta)]^{\frac{(q+1)(N-3)}{\gamma_0(N-1)}}\leq C.
\end{equation}
Therefore, if $N\leq3$ or $q<\frac{2}{N-3}$, we obtain this lemma for any $a\in[0,1)$ according to \eqref{decayforu11} and \eqref{decayforu12}. Hence in the following, we assume that $N\geq4$ and $q\geq\frac{2}{N-3}$. \par
Let $\mu=\frac{N-1}{N-3}\kappa=\frac{(N-1)(\alpha+\beta)}{(N-3)\alpha}$. Since $(p,q)$ is subcritical, by direct calculations we have
$$\alpha+\beta>N-2,~~\frac{\alpha+\beta+2}{\alpha+\beta}<\frac{N-1}{N-3}~~\mbox{and}~~q+1<\mu.$$
\par

For any $\gamma\in(0,1)$, let $t=t(\gamma)$ solve the equation
\begin{equation*}
\frac{1}{q+1}=\frac{N-3}{\gamma(N-1)}(1-t)+\frac{t}{\mu}.
\end{equation*}
Then
$$t=\frac{\frac{N-3}{\gamma(N-1)}-\frac{1}{q+1}}{\frac{N-3}{\gamma(N-1)}-\frac{1}{\mu}}\in (0,1).$$
By \eqref{decayforu1}, interpolation inequalities,  (1) of Lemma \ref{Sobolevimbedings}, Lemma \ref{lpesti}, Lemma \ref{esti}, Lemma \ref{decayfore102} and the fact that $r_0\in[1,2]$, we have
\begin{equation}\label{decayforu13}
\begin{aligned}
\|u(r_0)\|_{q+1}&\leq\|u(r_0)\|_{\frac{\gamma(N-1)}{N-3}}^{1-t}\|u(r_0)\|_\mu^t\\
&\leq C(\|D_\theta u^{\frac{\kappa}{2}}(r_0)\|_2^t+1)\\
&\leq C[(\int_{S_{r_0}}|\nabla u^{\frac{\kappa}{2}}|^2)^{\frac{t}{2}}+1)\\
&\leq C[F(4)^{\frac{t}{2}}+1].
\end{aligned}
\end{equation}
Hence by \eqref{decayforu13} and the fact that $r_0\in[1,2]$, we have
\begin{equation}\label{decayforu2}
\int_{S_{r_0}}u^{q+1}\leq C\int_{S_1}u^{q+1}(r_0,\theta)\leq C[F(4)^{a(\gamma)}+1],
\end{equation}
where $C=C(N,p,q,\gamma)>0$ and $a(\gamma)=\frac{(q+1)t}{2}$. Since $q\leq 1$ and $t\in (0,1)$,  for any $\gamma \in (0,1)$, we have
\begin{equation*}
a(\gamma)<1.
\end{equation*}
Therefore, choose any $\gamma_0\in(0,1)$ and let $a=a(\gamma_0)$ in \eqref{decayforu2}. Then, we obtain this lemma.
\end{proof}
$Proof~of~Theorem~\ref{mainthm1}.$ We only need to prove this theorem in the case of bounded solutions and $pq>1$.\par
Suppose that $q=\min\{p,q\}\leq1$, $(p,q)$ is subcritical and $(u,v)$ is a bounded positive solution of system \eqref{lesys}. Then $\max\{a, \frac{p+a}{p+1}, \frac{q+a}{q+1}\}=\frac{p+a}{p+1}$. By Lemma \ref{compair}, Lemma \ref{ph}, Lemma \ref{decayforgrand}, Lemma \ref{decayforu}, H\"{o}lder inequalities and the fact that $r_0\in [1,2]$, there exists $C=C(N,p,q)>0$ such that
\begin{equation*}
\begin{aligned}
&\int_{B_1}u^{q+1}\leq\int_{B_{r_0}}u^{q+1}\\
\leq& C\int_{S_{r_0}}(u^{q+1}+|\frac{\partial v}{\partial\nu}|u+|\frac{\partial u}{\partial\nu}|v)\\
\leq& C[F(4)^a+\|Dv\|_{L^{\frac{q+1}{q}}(S_{r_0})}\|u\|_{L^{q+1}(S_{r_0})}+\|Du\|_{L^{\frac{p+1}{p}}(S_{r_0})}\|v\|_{L^{p+1}(S_{r_0})}+1]\\
\leq& C[F(4)^{\frac{p+a}{p+1}}+1],
\end{aligned}
\end{equation*}
where $C=C(N,p,q)>0$.\par
Replacing $(u,v)$ by $(R^\alpha u(Rx),R^\beta v(Rx))$ in the above inequalities and changing variables yield that
\begin{equation*}
F(R)\leq C[F(4R)^{\frac{p+a}{p+1}}R^{\frac{(1-a)(N-2-\alpha-\beta)}{p+1}}+R^{N-2-\alpha-\beta}].
\end{equation*}
\par Since $u$ is bounded, there exist a large positive constant $c>0$ and a sequence of $\{R_i\}_{i=1}^{\infty}$ with $\lim \limits_{i\rightarrow \infty} R_i=+\infty$ such that $F(4R_i)\leq c F(R_i)$ (see \cite{CHL}, \cite{S}).
Therefore,
\begin{equation*}
F(R_i)\leq C[c^{\frac{p+a}{p+1}}F(R_i)^{\frac{p+a}{p+1}}R_i^{\frac{(1-a)(N-2-\alpha-\beta)}{p+1}}+R_i^{N-2-\alpha-\beta}].
\end{equation*}
Since $(p,q)$ is subcritical and $a\in[0,1)$, we have $\frac{p+a}{p+1}\in(0,1)$, $N-2-\alpha-\beta<0$ and $\frac{(1-a)(N-2-\alpha-\beta)}{p+1}<0$. Therefore, $F(R_i)$ is bounded, 
$$\int_{\mathbb{R}^N} u^{q+1}=\lim \limits_{i\rightarrow \infty} \int_{B_{R_i}}u^{q+1}=0,$$
and
$$u\equiv 0,~~ v\equiv 0,$$
which is a contradiction.

\section{Proof of Theorem \ref{mainthm}}
\numberwithin{equation}{section}
 \setcounter{equation}{0}
In this section, we prove Theorem \ref{mainthm}.

\begin{lem}\label{2017-09}
Suppose that $pq>1$, $(p,q)$ is subcritical, $\alpha=\max\{\alpha,\beta\}>N-3$ and $(u,v)$ is a bounded positive solution of system \eqref{lesys}. Then there exist $C=C(N,p,q)>0$ and $b=b(N,p,q)\in [0,1)$ such that
\begin{equation*}
\int_{S_{r_0}}u^{q+1}\leq C[F(4)^b+1].
\end{equation*}
\end{lem}

\begin{proof}
By the proof of Lemma \ref{decayforu}, we only need to consider the case of $N\geq4$ and $q\geq\frac{2}{N-3}$.\par
Let $l=\frac{p+1}{p}$ and $\frac{1}{\nu}=\frac{1}{l}-\frac{2}{N-1}$. Since $(p,q)$ is subcritical, then
\begin{equation*}\frac{1}{\nu}<\frac{1}{l}-\frac{2}{N}=1-\frac{1}{p+1}-\frac{2}{N}<\frac{1}{q+1}~~\mbox{and}~~q+1<\nu.
\end{equation*}\par
For any $\gamma\in(0,1]$, let $s=s(\gamma)\in [0,1)$ solve the equation
\begin{equation*}
\frac{1}{q+1}=\frac{N-3}{\gamma(N-1)}(1-s)+\frac{s}{\nu}.
\end{equation*}
Then
$$s=\frac{\frac{N-3}{\gamma(N-1)}-\frac{1}{q+1}}{\frac{N-3}{\gamma(N-1)}-\frac{1}{\nu}}~~\mbox{and}~~s(1)=(p+1)(\frac{N-3}{N-1}-\frac{1}{q+1}).$$\par
By \eqref{decayforu1}, interpolation inequalities,  (1) of Lemma \ref{Sobolevimbedings}, Lemma \ref{lpesti}, Lemma \ref{esti} and the fact that $r_0\in[1,2]$, we have
\begin{equation}\label{2017-09-01}
\begin{aligned}
\|u(r_0)\|_{q+1}&\leq\|u(r_0)\|_{\frac{\gamma(N-1)}{N-3}}^{1-s}\|u(r_0)\|_\nu^s\\
&\leq C(\|D_\theta^2u(r_0,\theta)\|_l^s+1)\\
&\leq C+C(\int_{S_{r_0}}|D^2u|^l)^{\frac{s}{l}}\\
&\leq C+C(\int_{B_2}|D^2u|^l)^{\frac{s}{l}}\\
&\leq C+C(\int_{B_4}|\Delta u|^l)^{\frac{s}{l}}\\
&=C+C(\int_{B_4}v^{p+1})^{\frac{s}{l}}.
\end{aligned}
\end{equation}\par
Since $(u,v)$ is bounded, by \eqref{2017-09-01}, Lemma \ref{compair} and the fact that $r_0\in[1,2]$, we have
\begin{equation}\label{2017-09-02}
\int_{S_{r_0}}u^{q+1}
\leq C+C(\int_{B_4}v^{p+1})^{b(\gamma)}\leq C[(\int_{B_4}u^{q+1})^{b(\gamma)}+1],
\end{equation}
where $C=C(N,p,q,\gamma)>0$ and $b(\gamma)=\frac{(q+1)s}{l}=\frac{p(q+1)s}{p+1}\in C((0,1])$. Then $b(1)=\frac{N-3}{N-1}p(q+1)-p$.\par
Since $(p,q)$ is subcritical, by \eqref{alphabeta} we have $p=\frac{\alpha+2}{\beta}$, $q=\frac{\beta+2}{\alpha}$ and since $p\geq q$ and $\alpha>N-3$, we have
\begin{equation*}
b(1)-1=(q+1)[\frac{(N-3)(\alpha+2)}{(N-1)\beta}-\frac{\alpha}{\beta}]=\frac{2(q+1)}{(N-1)\beta}(-\alpha+N-3)<0.
\end{equation*}
Therefore, there exists $\gamma_1<1$  such that $a(\gamma_1)<1$. Let $a=a(\gamma_1)$ in \eqref{2017-09-02}. Then, we obtain this lemma.
\end{proof}
By Lemma \ref{compair}, Lemma \ref{ph}, Lemma \ref{decayforgrand}, Lemma \ref{2017-09} and scale invariance of the solutions, we can prove Theorem \ref{mainthm}. Since the proof is similar to the proof of Theorem \ref{mainthm1}, we omit it.

\end{document}